\documentclass{amsart}
\usepackage{amsmath, amscd}
\usepackage{stmaryrd}
\usepackage{mathrsfs}
\usepackage{amssymb}
\usepackage{amsfonts}
\usepackage{tensor}
\usepackage[all]{xy}
\usepackage{enumerate}
\usepackage{tikz}
\usepackage{adjustbox}
\usepackage{graphicx}
\usetikzlibrary{matrix, arrows}
\usetikzlibrary{positioning}
\usetikzlibrary{decorations.pathreplacing}
\numberwithin{equation}{subsection}
\theoremstyle{plain}
\newtheorem{thm}[subsection]{Theorem}

\theoremstyle{definition}

\theoremstyle{remark}
\newtheorem{rem}[subsection]{Remark}

\newtheorem{final remark}[subsection]{Final Remark}
\makeatletter
\newcommand*\bigcdot{\mathpalette\bigcdot@{.65}}
\newcommand*\bigcdot@[2]{\mathbin{\vcenter{\hbox{\scalebox{#2}{$\m@th#1\bullet$}}}}}
\makeatother
\begin{document}
\title{Torsion codimension $2$ cycles on supersingular abelian varieties}
\author{Oliver Gregory}
\begin{abstract}
We prove that torsion codimension $2$ algebraic cycles modulo rational equivalence on supersingular abelian varieties are algebraically equivalent to zero. As a consequence, we prove that homological equivalence coincides with algebraic equivalence for algebraic cycles of codimension $2$ on supersingular abelian varieties over the algebraic closure of finite fields.
\end{abstract}
\address{Laver Building, University of Exeter, Exeter, EX4 4QF, Devon, UK}
\email {o.b.gregory@exeter.ac.uk}
\date{May 09, 2022 \\ This research was supported by EPSRC grant EP/T005351/1}
\maketitle
\pagestyle{myheadings}

\section{Introduction}

Let $k$ be a field and let $X$ be a smooth and projective variety over $k$. Write $Z^{n}(X)$ for the group of codimension $n$ algebraic cycles on $X$, and let $Z^{n}_{\mathrm{rat}}(X)\subseteq Z^{n}_{\mathrm{alg}}(X)\subseteq Z^{n}_{\mathrm{hom}}(X)$ be the subgroups of $Z^{n}(X)$ consisting of those cycles which are rationally, (resp. algebraically, resp. homologically (with respect to a fixed Weil cohomology theory - see Remark \ref{Weil independence})) equivalent to zero. Let $\mathrm{CH}^{n}(X)\supseteq\mathrm{CH}^{n}_{\mathrm{hom}}(X)\supseteq\mathrm{CH}^{n}_{\mathrm{alg}}(X)$ denote the quotients of $Z^{n}(X)\supseteq Z^{n}_{\mathrm{hom}}(X)\supseteq Z^{n}_{\mathrm{alg}}(X)$ by $Z_{\mathrm{rat}}(X)$. Let $\mathrm{Griff}^{n}(X):=\mathrm{CH}_{\mathrm{hom}}^{n}(X)/\mathrm{CH}_{\mathrm{alg}}^{n}(X)$ be the Griffiths group of codimension $n$ cycles on $X$.

Griffiths \cite{Gri69} was the first to show that smooth projective varieties can have non-trivial Griffiths groups - $\mathrm{Griff}^{2}(X)\otimes\mathbb{Q}$ is non-trivial for a very general quintic hypersurface $X\subset\mathbb{P}_{\mathbb{C}}^{4}$. Clemens \cite{Cle83} later showed that such hypersurfaces have $\mathrm{dim}_{\mathbb{Q}}(\mathrm{Griff}^{2}(X)\otimes\mathbb{Q})=\infty$, and then Voisin \cite{Voi00} generalised this by proving that $\mathrm{dim}_{\mathbb{Q}}(\mathrm{Griff}^{2}(X)\otimes\mathbb{Q})=\infty$ for any very general Calabi-Yau threefold $X$ over $\mathbb{C}$. Since we are interested in abelian varieties, let us also point out that Ceresa \cite{Cer83} has shown that for a very general curve $C$ of genus $\geq 3$, the Ceresa cycle is a non-trivial element in $\mathrm{Griff}^{2}(J(C))$. In fact, it was shown in \cite{Sch22} that the torsion subgroup of $\mathrm{Griff}^{n}$ need not even be finitely generated (at least for $n\geq 3$, with the situation for $n=2$ currently open). This phenomena is not specific to ``large'' base fields either: Harris \cite{Har83} has given an explicit abelian threefold defined over $\mathbb{Q}$ with non-trivial Griffiths group - the Ceresa cycle on the Jacobian of the Fermat quartic is not algebraically equivalent to zero. Bloch \cite[Theorem 4.1]{Bl84} gave a different proof of Harris' result on the Jacobian of the Fermat quartic, and showed moreover that the Ceresa cycle is non-torsion. In general, for a smooth projective variety defined over a number field, it is part of the Bloch-Beilinson conjectures that the dimensions of the Griffiths groups tensor $\mathbb{Q}$ are finite, and the dimensions are controlled by the orders of vanishing at integers of appropriate $L$-functions \cite[p. 381]{Bl85}.

In this note we are interested in the situation where the base field has positive characteristic. In this setting, again the Griffiths groups of smooth projective varieties can be non-trivial and even infinite. For example, Schoen \cite[Theorem 0.1]{Sch95} showed that if $k$ is a finite field of characteristic $p\equiv 1\mod3$ and $E$ denotes the Fermat cubic, then $\mathrm{Griff}^{2}(E^{3}_{\overline{k}})$ is non-trivial and has a non-trivial divisible part. Here $\overline{k}$ denotes the algebraic closure of $k$, and subscript $\overline{k}$ means the base change to $\overline{k}$.

Now let $k$ be a perfect field with $\mathrm{char}(k)=p>0$. The positive characteristic analogue of the Bloch-Beilinson philosophy says that cycles on varieties, at least after tensoring with $\mathbb{Q}$, are controlled by the slopes of the Frobenius on crystalline cohomology. Recall that a smooth proper variety $X$ over $k$ is called ordinary if $H^{m}(X,d\Omega^{r}_{X/k})=0$ for all $m,r$. When the crystalline cohomology groups $H_{\mathrm{cris}}^{n}(X/W(k))$ of $X$ are torsion-free, $X$ is ordinary if and only if, for each $n$, the Newton polygon of $X$ coincides with the Hodge polygon \cite[Proposition 7.3]{BK86}. If $A$ is an abelian variety, then $A$ is ordinary if and only if $A(\overline{k})[p]=(\mathbb{Z}/p\mathbb{Z})^{\dim A}$. For example, the condition that $p\equiv 1\mod 3$ in Schoen's theorem forces $E_{\overline{k}}^{3}$ to be an ordinary abelian threefold. Generalising Schoen's result, under the rubric of the Tate conjecture for surfaces over finite fields, Brent Gordon and Joshi \cite[Proposition 6.2]{BGJ02} proved that the codimension $2$ Griffiths group of ordinary abelian threefolds over the algebraic closure of a finite field are non-trivial, and contain a non-trivial divisible part.

At the opposite extreme to ordinarity is supersingularity, and in this situation the Bloch-Beilinson philosophy suggests that Griffiths group should be smaller because of the extreme degeneracy in the slopes of Frobenius. Recall that a smooth projective variety $X$ over $k$ is said to be supersingular if the Newton polygons of $X$ are isoclinic. If $A$ is an abelian variety, then $A$ is supersingular if and only $A_{\overline{k}}$ is isogenous to the self-product of a (any!) supersingular elliptic curve, where an elliptic curve $E$ is supersingular if and only if $E(\overline{k})[p]=0$ (see \cite[Theorem 4.2]{Oor74}). Schoen \cite[Theorem 14.4]{Sch95} showed that if $k$ is a finite field of characteristic $p\equiv 2\mod3$ and $E$ denotes the Fermat cubic, then $\mathrm{Griff}^{2}(E^{3}_{\overline{k}})$ is at most a $p$-primary torsion group. The condition that $p\equiv 2\mod 3$ implies that $E^{3}$ is a supersingular abelian threefold. Using work of Fakhruddin \cite{Fak02}, Brent Gordon and Joshi \cite[Theorem 5.1]{BGJ02} generalised Schoen's result to all supersingular abelian varieties - the codimension $2$ Griffiths groups of supersingular abelian varieties defined over the algebraic closure of a finite field are at most $p$-primary torsion. The question of whether these groups possess non-trivial $p$-torsion was left open. 

In this note we prove the following (see Theorem \ref{main} (1)):
\begin{thm}\label{main intro}
Let $k$ be an algebraically closed field of characteristic $p>0$, and let $A$ be a supersingular abelian variety over $k$. Then the inclusions
\begin{equation*}
\mathrm{CH}^{2}_{\mathrm{alg}}(A)_{tors}\subseteq\mathrm{CH}^{2}_{\mathrm{hom}}(A)_{tors}\subseteq\mathrm{CH}^{2}(A)_{tors}
\end{equation*}
are equalities. (Here $G_{tors}$ denotes the torsion subgroup of the group $G$).
\end{thm}
It was already shown in the proof of \cite[Theorem 5.1]{BGJ02} that $\mathrm{CH}^{2}_{\mathrm{alg}}(A)[\ell^{\infty}]=\mathrm{CH}^{2}_{\mathrm{hom}}(A)[\ell^{\infty}]=\mathrm{CH}^{2}(A)[\ell^{\infty}]$ for each prime $\ell\neq p$ (where $G[\ell^{\infty}]$ denotes the $\ell$-primary torsion subgroup of the group $G$). Our only new result is that this is also true for $p$-primary torsion. To handle the $\ell=p$ case, we initially follow the proof of Brent Gordon and Joshi, but then conclude using an inductive argument based on the Bloch-Srinivas method \cite{BS83}.

As a consequence of Theorem \ref{main intro}, we settle the $p$-primary torsion case of \cite[Theorem 5.1]{BGJ02}. Indeed, we have the following corollary (see Theorem \ref{main} (2)):
\begin{thm}
Let $k$ be a finite field of characteristic $p>0$, and let $A$ be a supersingular abelian variety over $\overline{k}$. Then $\mathrm{Griff}^{2}(A)[p^{\infty}]$ is trivial.
\end{thm}
Together with \cite[Theorem 5.1]{BGJ02}, this shows that $\mathrm{Griff}^{2}(A)$ is trivial. That is, homological equivalence coincides with algebraic equivalence for codimension $2$ cycles on supersingular abelian varieties over the algebraic closure of finite fields.
\\
\\
\emph{Acknowledgements:} The author would like to thank Kirti Joshi for his generosity in sharing his ideas, and for enlightening discussions.

\section{Chow groups of supersingular abelian varieties}\label{cycles on ab vars}

We repeat the discussion from \cite[\S2 and \S3]{BGJ02}. Let $A$ be an abelian variety of dimension $g$ over a field $k$, and let $n$ be a non-negative integer. Then by work of Mukai \cite{Muk81}, Beauville \cite{Bea86} and Deninger-Murre \cite{DM91}, the rational Chow groups of $A$ admit a direct sum decomposition
\begin{equation*}
\mathrm{CH}^{n}(A)\otimes\mathbb{Q}=\bigoplus_{i}\mathrm{CH}^{n}_{i}(A)
\end{equation*}
where $\mathrm{CH}^{n}_{i}(A):=\{Z\in\mathrm{CH}^{n}(A)\otimes\mathbb{Q}\,:\,m_{A}^{\ast}(Z)=m^{2n-i}Z\text{ for all }m\in\mathbb{Z}\}$ and $m_{A}^{\ast}$ denotes the flat pullback by multiplication-by-$m$ on $A$. 

Now suppose that $k$ is algebraically closed of characteristic $p>0$. Then Fakhruddin \cite{Fak02} has proved that if $A$ is a supersingular abelian variety over $k$, $\mathrm{CH}^{n}_{i}(A)=0$ for $i\neq 0,1$. Moreover,  the $\ell$-adic cycle class map induces an isomorphism 
\begin{equation*}
\mathrm{CH}^{n}_{0}(A)\otimes\mathbb{Q}_{\ell}\xrightarrow{\sim}H^{2n}_{\mathrm{\acute{e}t}}(A,\mathbb{Q}_{\ell}(n))
\end{equation*}
for all primes $\ell\neq p$. The same proof shows that the crystalline cycle class map induces an isomorphism 
\begin{equation*}
\mathrm{CH}^{n}_{0}(A)\otimes K\xrightarrow{\sim}H_{\mathrm{cris}}^{2n}(A/W(k))\otimes_{W(k)}K
\end{equation*}	
where $K=W(k)[1/p]$ is the fraction field of the Witt vectors $W(k)$ of $k$. In particular, $\mathrm{CH}^{n}_{1}(A)=\mathrm{CH}^{n}_{\mathrm{hom}}(A)\otimes\mathbb{Q}$ where $\mathrm{CH}^{n}_{\mathrm{hom}}(A)$ is the kernel of the cycle class map.
\begin{rem}\label{Weil independence}
A priori, the definition of the group $\mathrm{CH}^{n}_{\mathrm{hom}}(X)$ of codimension $n$ cycles homologically equivalent to zero on a smooth projective variety $X$ depends on the choice of Weil cohomology theory for $X$. Of course, it is a consequence of the standard conjectures (specifically that homological equivalence coincides with numerical equivalence) that $\mathrm{CH}^{n}_{\mathrm{hom}}(X)$ is independent of the choice of Weil cohomology theory. Notice, though, that Fakhruddin's result shows that $\mathrm{CH}^{n}_{\mathrm{hom}}(A)$ is independent of any choice when $A$ is a supersingular variety over an algebraically closed field of characteristic $p>0$. Since this is the setting that we are interested in, there is no ambiguity in the definition.
\end{rem}

As pointed out in \cite[\S2]{BGJ02}, if $k$ is moreover the algebraic closure of a finite field, then it is known by results of Soul\'{e} \cite{Sou84} and K\"{u}nnemann \cite{Kun93} that $\mathrm{CH}^{n}_{1}(A)=0$. In particular, $\mathrm{CH}^{n}_{\mathrm{hom}}(A)$ is torsion.

\begin{rem}
Beilinson \cite[1.0]{Bei87} has conjectured that $\mathrm{CH}^{n}_{\mathrm{hom}}(X)$ is torsion for any smooth projective variety $X$ over the algebraic closure of a finite field. 
\end{rem}

\section{Abel-Jacobi maps} 

In this section we fix notation involving $\ell$-adic Abel-Jacobi maps, for primes $\ell$ (including $\ell=p$). 

Let $X$ be a smooth projective variety over an algebraically closed field $k$ of characteristic $p\geq 0$. Let $\ell$ be a prime. Define
\begin{equation*}
H^{i}(X,\mathbb{Z}_{\ell}(j)):=\begin{cases} 
      H^{i}_{\mathrm{\acute{e}t}}(X,\mathbb{Z}_{\ell}(j)) & \text{if }\ell\neq p \\
      H^{i-j}(X_{\mathrm{\acute{e}t}},W\Omega^{j}_{X,\log}) & \text{if }\ell= p
   \end{cases}
\end{equation*}
and
\begin{equation*}
H^{i}(X,\mathbb{Q}_{\ell}/\mathbb{Z}_{\ell}(j)):=\begin{cases} 
      H^{i}_{\mathrm{\acute{e}t}}(X,\mathbb{Q}_{\ell}/\mathbb{Z}_{\ell}(j)) & \text{if }\ell\neq p \\
      \displaystyle\varinjlim_{r}H^{i-j}(X_{\mathrm{\acute{e}t}},W_{r}\Omega^{j}_{X,\log}) & \text{if }\ell= p
   \end{cases}
\end{equation*}
where $W_{r}\Omega^{j}_{X,\log}$ denotes the logarithmic Hodge-Witt sheaf of $X$ (see \cite[Ch. I, 5.7]{Ill79}) and the limit is taken over the maps $\underline{p}:W_{r}\Omega^{j}_{X}\rightarrow W_{r+1}\Omega^{j}_{X}$ \cite[Ch. I, Proposition 3.4]{Ill79}. Let 
\begin{equation*}
\lambda_{n,X}:\mathrm{CH}^{n}(X)[\ell^{\infty}]\rightarrow H^{2n-1}(X,\mathbb{Q}_{\ell}/\mathbb{Z}_{\ell}(n))
\end{equation*}
be Bloch's $\ell$-adic Abel-Jacobi map \cite{Blo79} if $\ell\neq p$, and the Gros-Suwa $p$-adic Abel-Jacobi map \cite{GS88} if $\ell=p$. Here $\mathrm{CH}^{n}(X)[\ell^{\infty}]$ denotes the $\ell$-primary torsion subgroup of $\mathrm{CH}^{n}(X)$.

\section{The result}

\begin{thm}\label{main}
Let $k$ be an algebraically closed field of characteristic $p>0$, and let $A$ be a supersingular abelian variety over $k$. Then:
\begin{enumerate}
\item We have 
\begin{equation*}
\mathrm{CH}^{2}_{\mathrm{alg}}(A)_{tors}=\mathrm{CH}^{2}_{\mathrm{hom}}(A)_{tors}=\mathrm{CH}^{2}(A)_{tors}\,.
\end{equation*}

\item If $k$ is the algebraic closure of a finite field, then $\mathrm{Griff}^{2}(A)$ is trivial, i.e. 
\begin{equation*}
\mathrm{CH}^{2}_{\mathrm{alg}}(A)=\mathrm{CH}^{2}_{\mathrm{hom}}(A)\,.
\end{equation*}
\end{enumerate}
\end{thm}
\begin{proof}
(1) Let $\ell$ be a prime. Consider the following commutative diagram
\begin{equation}\label{inclusions}
\begin{tikzpicture}[descr/.style={fill=white,inner sep=1.5pt}]
        \matrix (m) [
            matrix of math nodes,
            row sep=2.5em,
            column sep=2.5em,
            text height=1.5ex, text depth=0.25ex
        ]
        { \mathrm{CH}^{2}(A)[\ell^{\infty}] & H^{3}(A,\mathbb{Q}_{\ell}/\mathbb{Z}_{\ell}(2)) \\
       \mathrm{CH}_{\mathrm{alg}}^{2}(A)[\ell^{\infty}]  & \ \\};

        \path[overlay,->, font=\scriptsize] 
        (m-1-1) edge node[above]{$\lambda_{2,A}$}(m-1-2)
        (m-2-1) edge node[below right]{$\lambda'_{2,A}$} (m-1-2)
        ;
        
        \path[overlay, right hook->, font=\scriptsize]
        (m-2-1) edge (m-1-1)
        ;
                                        
\end{tikzpicture}
\end{equation}
where $\lambda'_{2,A}$ denotes the restriction of $\lambda_{2,A}$. It is known that $\lambda_{2,A}$ is injective (as a consequence of the Merkurjev-Suslin theorem \cite[Corollary 4]{CTSS83} for $\ell\neq p$, \cite[\S III Proposition 3.4]{GS88} for $\ell=p$), hence $\lambda'_{2,A}$ is injective. It therefore suffices to show that $\lambda'_{2,A}$ is surjective, since then all maps in \eqref{inclusions} are isomorphisms and in particular
\begin{equation*}
\mathrm{CH}^{2}_{\mathrm{alg}}(A)[\ell^{\infty}]=\mathrm{CH}^{2}_{\mathrm{hom}}(A)[\ell^{\infty}]=\mathrm{CH}^{2}(A)[\ell^{\infty}]
\end{equation*}
for each prime $\ell$.

For any $n\geq 0$, consider the following diagram
\begin{equation*}
\begin{adjustbox}{width=12cm}
\begin{tikzpicture}[descr/.style={fill=white,inner sep=1.5pt}]
        \matrix (m) [
            matrix of math nodes,
            row sep=2.5em,
            column sep=2.5em,
            text height=1.5ex, text depth=0.25ex
        ]
        { & & \mathrm{CH}^{n}(A)[\ell^{\infty}] &  \\
       0 & H^{2n-1}(A,\mathbb{Z}_{\ell}(n))\otimes_{\mathbb{Z}_{\ell}}\mathbb{Q}_{\ell}/\mathbb{Z}_{\ell} & H^{2n-1}(A,\mathbb{Q}_{\ell}/\mathbb{Z}_{\ell}(n)) & H^{2n}(A,\mathbb{Z}_{\ell}(n)) \\};

        \path[overlay,->, font=\scriptsize] 
        (m-2-1) edge (m-2-2)
        (m-2-2) edge (m-2-3)
        (m-2-3) edge (m-2-4)
        (m-1-3) edge node[right]{$\lambda_{n,A}$}(m-2-3)
        ;

\end{tikzpicture}
\end{adjustbox}
\end{equation*}
Up to a sign, the induced map $\mathrm{CH}^{n}(A)[\ell^{\infty}]\rightarrow H^{2n}(A,\mathbb{Z}_{\ell}(n))$ is the restriction of the cycle class map (\cite[Corollary 4]{CTSS83} for $\ell\neq p$, \cite[\S III Propositions 1.16 and 1.21]{GS88} for $\ell=p$). The bottom row of the diagram is exact (see \cite[(3.33)]{GS88} for exactness when $\ell=p$). Therefore the restriction of $\lambda_{n,A}$ to $\mathrm{CH}^{n}_{\mathrm{hom}}(A)[\ell^{\infty}]$ has image in $H^{2n-1}(A,\mathbb{Z}_{\ell}(n))\otimes_{\mathbb{Z}_{\ell}}\mathbb{Q}_{\ell}/\mathbb{Z}_{\ell}$. In particular, $\lambda'_{2,A}$ has image in $H^{3}(A,\mathbb{Z}_{\ell}(2))\otimes_{\mathbb{Z}_{\ell}}\mathbb{Q}_{\ell}/\mathbb{Z}_{\ell}$. Therefore the cokernel of $\lambda'_{2,A}$ is divisible. Note that $H^{3}(A,\mathbb{Z}_{\ell}(2))$ is torsion-free (the non-trivial case when $\ell=p$ follows from $H^{3}_{\mathrm{cris}}(A/W(k))$ being torsion-free \cite[Lemme 3.12]{GS88}, which it is because $H^{1}_{\mathrm{cris}}$ is always torsion-free and $H^{3}_{\mathrm{cris}}=\wedge^{3}H^{1}_{\mathrm{cris}}$ for abelian varieties), hence $H^{3}(A,\mathbb{Z}_{\ell}(2))\otimes_{\mathbb{Z}_{\ell}}\mathbb{Q}_{\ell}/\mathbb{Z}_{\ell}$ is a direct sum of a finite number of copies of $\mathbb{Q}_{\ell}/\mathbb{Z}_{\ell}$.

We are therefore reduced to showing that $\mathrm{coker}(\lambda'_{2,A})$ is annihilated by a positive integer. (Indeed, $\mathrm{coker}(\lambda'_{2,A})$ is a quotient of the divisible group $H^{3}(A,\mathbb{Z}_{\ell}(2))\otimes_{\mathbb{Z}_{\ell}}\mathbb{Q}_{\ell}/\mathbb{Z}_{\ell}\cong \left(\mathbb{Q}_{\ell}/\mathbb{Z}_{\ell}\right)^{r}$ for some $r$. If $\mathrm{coker}(\lambda'_{2,A})$ is finite then it must be trivial since any finite divisible group is trivial. So we must rule out the case that $\mathrm{coker}(\lambda'_{2,A})$ is infinite, in which case it is a finite number of copies of $\mathbb{Q}_{\ell}/\mathbb{Z}_{\ell}$. But this group is not annihilated by a positive integer.) We shall prove this by induction on the dimension $g$ of $A$. Of course, the entire theorem is trivial if $g=1$, so suppose that $g>1$ and suppose that $\lambda'_{2,B}$ is surjective for supersingular abelian varieties of dimension $\leq g-1$. By \cite[Lemma 3]{Fak02} and its proof, there exist $g$-dimensional abelian subvarieties $Y_{1},\ldots, Y_{n}\subset A\times A$ such that the class of the diagonal $\Delta_{A}$ decomposes as
\begin{equation*}
\Delta_{A}=\sum_{i}c_{i}[Y_{i}]\in\mathrm{CH}^{g}(A\times A)\otimes\mathbb{Q}
\end{equation*}
for some $c_{i}\in\mathbb{Q}$, and such that for each $i$, the image of $Y_{i}$ under at least one of the projections $\mathrm{pr}_{1},\mathrm{pr}_{2}:A\times A\rightarrow A$ has dimension $\leq g-1$ . By clearing denominators, we see that there exists an integer $N>0$ such that
\begin{equation*}
N\Delta_{A}=\sum_{i}d_{i}[Y_{i}]\in\mathrm{CH}^{g}(A\times A)
\end{equation*}
for some $d_{i}\in\mathbb{Z}$. Label the $Y_{1},\ldots, Y_{n}$ so that $A_{i}:=\mathrm{pr}_{1}(Y_{i})$ has dimension $\leq g-1$ for $i=1,\ldots, m$, and $A_{i}:=\mathrm{pr}_{2}(Y_{i})$ has dimension $\leq g-1$ for $i=m+1,\ldots, n$. Let $V_{1}:=A_{1}\cup\cdots\cup A_{m}$ and $V_{2}:=A_{m+1}\cup\cdots\cup A_{n}$, and let $j_{1},j_{2}:V_{1},V_{2}\hookrightarrow A$ be the natural inclusions. Then
\begin{equation*}
N\Delta_{A}=Z_{1}+Z_{2}\in\mathrm{CH}^{g}(A\times A)
\end{equation*}
where $\mathrm{Supp}(Z_{1})\subset V_{1}\times A$ and $\mathrm{Supp}(Z_{2})\subset A\times V_{2}$. Let $\widetilde{V}_{1}:=A_{1}\sqcup\cdots\sqcup A_{m}$ and $\widetilde{V}_{2}:=A_{m+1}\sqcup\cdots\sqcup A_{n}$ be the disjoint unions, and let $\tau_{1}:\widetilde{V}_{1}\rightarrow V_{1}$, $\tau_{2}:\widetilde{V}_{2}\rightarrow V_{2}$ be the natural morphisms. We claim that there is a correspondence $\widetilde{Z}_{1}\in\mathrm{CH}^{g-1}(\widetilde{V}_{1}\times A)$ such that
\begin{equation*}
Z_{1}=\widetilde{Z}_{1}\circ\gamma_{1}
\end{equation*}
where $\gamma_{1}\in\mathrm{CH}^{g}(A\times\widetilde{V}_{1})$ is the correspondence given by the transpose of the graph of $\widetilde{j}_{1}:=j_{1}\circ \tau_{1}$. Indeed, let $V_{1}^{\mathrm{sm}}$ denote the smooth locus of $V_{1}$, and consider the pullback square
\begin{equation*}
\begin{tikzpicture}[descr/.style={fill=white,inner sep=1.5pt}]
        \matrix (m) [
            matrix of math nodes,
            row sep=2.5em,
            column sep=3.5em,
            text height=1.5ex, text depth=0.25ex
        ]
        { \widetilde{V}_{1}\times A & V_{1}\times A\\
   \tau_{1}^{-1}(V_{1}^{\mathrm{sm}})\times A & V_{1}^{\mathrm{sm}}\times A \\
    };

        \path[overlay,->, font=\scriptsize] 
        (m-1-1) edge node[above]{$\tau_{1}\times\mathrm{id}_{A}$} (m-1-2)
        (m-2-1) edge node[above]{$\tau_{1}\times\mathrm{id}_{A}$} (m-2-2)
        ;
                                  
        \path[overlay, right hook->, font=scriptsize]
        (m-2-1) edge (m-1-1)
        (m-2-2) edge (m-1-2)
        ;
                                         
\end{tikzpicture}
\end{equation*}
Since $V_{1}^{\mathrm{sm}}$, $\tau_{1}^{-1}(V_{1}^{\mathrm{sm}})$ and $\widetilde{V}_{1}$ are smooth, the morphisms in the diagram admit refined Gysin pullbacks (see \cite[\S6.6]{Ful84}). Set $\widetilde{Z}_{1}$ to be the closure in $\widetilde{V}_{1}\times A$ of the pullback of $Z_{1}$ along $\tau_{1}^{-1}(V_{1}^{\mathrm{sm}})\times A\rightarrow V_{1}^{\mathrm{sm}}\times A\rightarrow V_{1}\times A$, where consider $Z_{1}$ as a cycle on $V_{1}\times A$. Then $Z_{1}=(j_{1}\circ\tau_{1}\times\mathrm{id}_{A})_{\ast}\widetilde{Z}_{1}=\widetilde{Z}_{1}\circ\gamma_{1}$ by \cite[Proposition 16.1.1]{Ful84}, as desired. The same argument applied to the transpose of $Z_{2}$ shows that there exists a correspondence $\widetilde{Z}_{2}\in\mathrm{CH}^{g-1}(A\times\widetilde{V}_{2})$ such that
\begin{equation*}
Z_{2}=\Gamma_{\widetilde{j}_{2}}\circ\widetilde{Z}_{2}
\end{equation*}
where $\Gamma_{\widetilde{j}_{2}}\in\mathrm{CH}^{g}(\widetilde{V}_{2}\times A)$ is the correspondence given by the graph of $\widetilde{j}_{1}:=j_{2}\circ \tau_{2}$. Hence, 
\begin{equation*}
N\Delta_{A}=Z_{1}+Z_{2}=\widetilde{Z}_{1}\circ\gamma_{1}+\Gamma_{\widetilde{j}_{2}}\circ\widetilde{Z}_{2}
\end{equation*}
and the self-correspondence $N\Delta_{A}^{\ast}:\mathrm{CH}^{2}(A)\rightarrow\mathrm{CH}^{2}(A)$ factors as
\begin{equation*}
\mathrm{CH}^{2}(A)\xrightarrow{\widetilde{Z}_{1}^{\ast}\oplus\widetilde{j}_{2}^{\ast}}\mathrm{CH}^{1}(\widetilde{V}_{1})\oplus\mathrm{CH}^{1}(\widetilde{V}_{2})\xrightarrow{\widetilde{j}_{1\ast}+\widetilde{Z}_{2}^{\ast}}\mathrm{CH}^{2}(A)\,.
\end{equation*}
Since the $\ell$-adic Abel-Jacobi maps are compatible with correspondences (\cite[Proposition 3.5]{Blo79} for $\ell\neq p$, \cite[Proposition 2.9]{GS88} for $\ell=p$), we get a commutative diagram 
  
\begin{equation}\label{correspondence}
\begin{adjustbox}{width=12cm}
\begin{tikzpicture}[descr/.style={fill=white,inner sep=1.5pt}]
        \matrix (m) [
            matrix of math nodes,
            row sep=2.5em,
            column sep=3.5em,
            text height=1.5ex, text depth=0.25ex
        ]
        { \mathrm{CH}^{2}_{\mathrm{alg}}(A)[\ell^{\infty}] & H^{3}(A,\mathbb{Z}_{\ell}(2))\otimes_{\mathbb{Z}_{\ell}}\mathbb{Q}_{\ell}/\mathbb{Z}_{\ell} \\
     \mathrm{CH}^{1}_{\mathrm{alg}}(\widetilde{V}_{1})[\ell^{\infty}]\oplus\mathrm{CH}^{2}_{\mathrm{alg}}(\widetilde{V}_{2})[\ell^{\infty}]   & H^{1}(\widetilde{V}_{1},\mathbb{Z}_{\ell}(1))\otimes_{\mathbb{Z}_{\ell}}\mathbb{Q}_{\ell}/\mathbb{Z}_{\ell}\oplus H^{3}(\widetilde{V}_{2},\mathbb{Z}_{\ell}(2)))\otimes_{\mathbb{Z}_{\ell}}\mathbb{Q}_{\ell}/\mathbb{Z}_{\ell} \\
      \mathrm{CH}^{2}_{\mathrm{alg}}(A)[\ell^{\infty}] & H^{3}(A,\mathbb{Z}_{\ell}(2))\otimes_{\mathbb{Z}_{\ell}}\mathbb{Q}_{\ell}/\mathbb{Z}_{\ell} \\};

        \path[overlay,->, font=\scriptsize] 
        (m-1-1) edge node[left]{$\widetilde{Z}_{1}^{\ast}\oplus\widetilde{j}_{2}^{\ast}$} (m-2-1)
         (m-1-2) edge node[left]{$\widetilde{Z}_{1}^{\ast}\oplus\widetilde{j}_{2}^{\ast}$} (m-2-2)
          (m-2-1) edge node[left]{$\widetilde{j}_{1\ast}+\widetilde{Z}_{2}^{\ast}$} (m-3-1)
        (m-2-2) edge node[left]{$\widetilde{j}_{1\ast}+\widetilde{Z}_{2}^{\ast}$} (m-3-2)
        (m-1-1) edge node[above]{$\lambda'_{2,A}$} (m-1-2)
        (m-2-1) edge node[above]{$\lambda'_{1,\widetilde{V}_{1}}\oplus\lambda'_{2,\widetilde{V}_{2}}$} (m-2-2)
        (m-3-1) edge node[above]{$\lambda'_{2,A}$} (m-3-2)
        ;
                                                 
\end{tikzpicture}
\end{adjustbox}
\end{equation}
where the composition of the vertical arrows is $N\Delta_{A}^{\ast}$. But $\Delta_{A}^{\ast}$ is the identity, so $N\Delta_{A}^{\ast}$ is multiplication-by-$N$.

The map $\lambda'_{1,\widetilde{V}_{1}}:\mathrm{CH}^{1}_{\mathrm{alg}}(\widetilde{V}_{1})[\ell^{\infty}]\rightarrow H^{1}(\widetilde{V}_{1},\mathbb{Z}_{\ell}(1))\otimes_{\mathbb{Z}_{\ell}}\mathbb{Q}_{\ell}/\mathbb{Z}_{\ell}$ is a bijection \cite[Proposition A.28]{ACMV21}. We claim that the map $\lambda'_{2,\widetilde{V}_{2}}:\mathrm{CH}^{2}_{\mathrm{alg}}(\widetilde{V}_{2})[\ell^{\infty}]\rightarrow H^{3}(\widetilde{V}_{2},\mathbb{Z}_{\ell}(2))\otimes_{\mathbb{Z}_{\ell}}\mathbb{Q}_{\ell}/\mathbb{Z}_{\ell}$ is also a bijection. Indeed, it is injective by the same reasoning that showed $\lambda'_{2,A}$ is injective. To see that $\lambda'_{2,\widetilde{V}_{2}}$ is surjective, recall that $\widetilde{V}_{2}:=A_{m+1}\sqcup\cdots\sqcup A_{n}$ and $\lambda'_{2,\widetilde{V}_{2}}$ is the direct sum
\begin{equation*}
\bigoplus_{i=m+1}^{n}\mathrm{CH}^{2}_{\mathrm{alg}}(A_{i})[\ell^{\infty}]\xrightarrow{\oplus\lambda'_{2,A_{i}}}\bigoplus_{i=m+1}^{n}H^{3}(A_{i},\mathbb{Z}_{\ell}(2))\otimes{\mathbb{Z}_{\ell}}\mathbb{Q}_{\ell}/\mathbb{Z}_{\ell}\,.
\end{equation*}
The $A_{i}$ are supersingular abelian varieties of dimension $\leq g-1$ (they are supersingular since they are subvarieties of $A$), so the induction hypothesis implies that $\lambda'_{2,\widetilde{V}_{2}}=\oplus\lambda'_{2,A_{i}}$ is surjective as claimed.

In particular, we see that the middle horizontal arrow in \eqref{correspondence} is a bijection. A diagram chase shows that $\mathrm{coker}(\lambda'_{2,A})$ is annihilated by $N$.

(2) We have seen in \S\ref{cycles on ab vars} that $\mathrm{CH}^{2}_{\mathrm{hom}}(A)$ is a torsion group when $k$ is the algebraic closure of a finite field. Therefore the subgroup $\mathrm{CH}^{2}_{\mathrm{alg}}(A)$ is also torsion. We may then conclude by part (1).
\end{proof}

\begin{rem}
It was already shown in \cite[Theorem 5.1]{BGJ02} that $\mathrm{Griff}^{2}(A)[\ell^{\infty}]=0$ for all primes $\ell\neq p$, so the only new result is that $\mathrm{Griff}^{2}(A)[p^{\infty}]=0$ as well. The proof strategy in Theorem \ref{main} of reducing to showing surjectivity of $\lambda'_{2,A}$ is the same as \cite[Theorem 5.1]{BGJ02}. When $\ell\neq p$, surjectivity of $\lambda'_{2,A}$ is due to Suwa \cite[Th\'{e}or\`{e}me 4.7.1]{Suw88}. Suwa's proof proceeds by considering the following commutative diagram
\begin{equation}\label{ell square}
\begin{tikzpicture}[descr/.style={fill=white,inner sep=1.5pt}]
        \matrix (m) [
            matrix of math nodes,
            row sep=2.5em,
            column sep=3.5em,
            text height=1.5ex, text depth=0.25ex
        ]
        { \mathrm{CH}^{1}(A)\otimes\mathbb{Z}_{\ell}\times\mathrm{CH}^{1}_{\mathrm{alg}}(A)[\ell^{\infty}] & \mathrm{CH}^{2}_{\mathrm{alg}}(A)[\ell^{\infty}]\\
   H^{2}(A,\mathbb{Z}_{\ell}(1))\times H^{1}(A,\mathbb{Q}_{\ell}/\mathbb{Z}_{\ell}(1)) & H^{3}(A,\mathbb{Q}_{\ell}/\mathbb{Z}_{\ell}(2)) \\
    };

        \path[overlay,->, font=\scriptsize] 
        (m-1-1) edge node[left]{$\mathrm{cl}\times\lambda'_{1,A}$} (m-2-1)
        (m-1-2) edge node[right]{$\lambda'_{2,A}$} (m-2-2)
        (m-1-1) edge node[above]{$\cdot$} (m-1-2)
        (m-2-1) edge node[above]{$\cup$} (m-2-2)
        ;
                                                 
\end{tikzpicture}
\end{equation}
The cycle class map $\mathrm{cl}$ is surjective for supersingular abelian varieties by \cite[Appendix]{Shi75}, and we have already seen that $\lambda'_{1,A}$ is a bijection. The cup-product map along the bottom of the square is a surjection as a consequence of $H^{i}(A,\mathbb{Z}/\ell\mathbb{Z})=\bigwedge^{i}H^{1}(A,\mathbb{Z}/\ell\mathbb{Z})$ for abelian varieties. This forces $\lambda'_{2,A}$ to be surjective.

In the case $\ell=p$, we have the commutative square analogous to \eqref{ell square}. Unravelling notation, the square is as follows
 \begin{equation*}
\begin{tikzpicture}[descr/.style={fill=white,inner sep=1.5pt}]
        \matrix (m) [
            matrix of math nodes,
            row sep=2.5em,
            column sep=3.5em,
            text height=1.5ex, text depth=0.25ex
        ]
        { \mathrm{CH}^{1}(A)\otimes\mathbb{Z}_{p}\times\mathrm{CH}^{1}_{\mathrm{alg}}(A)[p^{\infty}] & \mathrm{CH}^{2}_{\mathrm{alg}}(A)[p^{\infty}]\\
   H^{1}(A_{\mathrm{\acute{e}t}},W\Omega_{A,\log}^{1})\times\displaystyle\varinjlim_{r}H^{0}(A_{\mathrm{\acute{e}t}},W_{r}\Omega_{A,\log}^{1}) & \displaystyle\varinjlim_{r}H^{1}(A_{\mathrm{\acute{e}t}},W_{r}\Omega_{A,\log}^{2}) \\
    };

        \path[overlay,->, font=\scriptsize] 
        (m-1-1) edge node[left]{$\mathrm{cl}\times\lambda'_{1,A}$} (m-2-1)
        (m-1-2) edge node[right]{$\lambda'_{2,A}$} (m-2-2)
        (m-1-1) edge node[above]{$\cdot$} (m-1-2)
        (m-2-1) edge node[above]{$\cup$} (m-2-2)
        ;
                                                 
\end{tikzpicture}
\end{equation*}
The cup-product map along the bottom horizontal is rarely surjective when $A$ is not an ordinary abelian variety. Indeed, we have $\mathrm{CH}^{1}_{\mathrm{alg}}(A)=\mathrm{Pic}^{0}_{A/k}(k)$, so if $A$ is an abelian variety with $p$-rank $0$ (if $A$ is supersingular, for example) then $\mathrm{CH}^{1}_{\mathrm{alg}}(A)[p^{\infty}]\cong\varinjlim_{r}H^{0}(A_{\mathrm{\acute{e}t}},W_{r}\Omega_{A,\log}^{1})$ is the trivial group. This is why we must use a different argument for the $\ell=p$ case of Theorem \ref{main} than the argument for $\ell\neq p$ used in \cite[Theorem 5.1]{BGJ02}. Notice that the proof of Theorem \ref{main} treats all primes $\ell$ ( including $\ell=p$), and in particular gives a new proof of \cite[Theorem 5.1]{BGJ02}. 
\end{rem}

\begin{rem}
The proof of Theorem \ref{main} shows that the inclusion
\begin{equation*}
H^{3}(A,\mathbb{Z}_{\ell}(2))\otimes_{\mathbb{Z}_{\ell}}\mathbb{Q}_{\ell}/\mathbb{Z}_{\ell}\hookrightarrow H^{3}(A,\mathbb{Q}_{\ell}/\mathbb{Z}_{\ell}(2))
\end{equation*}
is an equality for all primes $\ell$. This was known for $\ell\neq p$ as a consequence of the proof of \cite[Theorem 5.1]{BGJ02}.
\end{rem}

\end{document}